\documentclass[12pt]{amsart}
\usepackage{amsmath,amssymb,amsthm}
\usepackage[all]{xy}

\def\P{{\mathbb{P}}}
\def\f{{\mathcal{F}}}

\def\p{{\mathfrak{p}}}

\def\C{\mathbb{C}}
\def\Z{\mathbb{Z}}
\def\Q{\mathbb{Q}}

\def\ord{{\rm ord}}

\def\GL{{\rm GL}}
\def\GSp{{\rm GSp}}

\newcommand{\inflim}[1]{\displaystyle \lim_{#1 \rightarrow \infty}}

\def\sep{{\rm sep}}

\def\Z{{\mathbb Z}}

\newcommand{\Fqbar}{\overline{\mathbb{F}}_q}

\newcommand\Fp{\mathbb{F}_p}

\def\P{{\bf{P}}}

\newenvironment{Gal}{{\rm Gal\,}}{}

\newenvironment{Disc}{{\rm Disc \,}}{}

\newenvironment{Aut}{{\rm Aut}}{}

\newenvironment{Per}{{\rm Per}}{}

\newtheorem{theorem}{Theorem}
\newtheorem{definition}[theorem]{Definition}
\newtheorem{lemma}[theorem]{Lemma}
\newtheorem{proposition}[theorem]{Proposition}
\newtheorem{proposition-definition}[theorem]{Proposition-Definition}

\newtheorem{conjecture}[theorem]{Conjecture}

\newtheorem{question}[theorem]{Question}

\theoremstyle{definition}

\theoremstyle{remark}
\newtheorem*{remark}{Remark}

\begin{document}

\title[2007: An arboreal odyssey]{2007: An arboreal odyssey \\ \smallskip \small{A view of arboreal Galois representations and applications, from early in the subject's history}}
\author{Rafe Jones}

\begin{abstract}
The study of arboreal Galois representations (that is, Galois groups arising from iteration of polynomial and rational functions) originated with work of Odoni in the 1980s. Beginning in the early 2000s it underwent a period of renewed interest, which continues to this day. Written in 2007, this survey article gives a sense of the subject from the early days of this renewal.  It is presented here as a document of historical interest -- precisely as originally written -- and because some recent work has referenced specific pieces of it. It was written as an informal document, and not intended to be published. Much, though not all, of the content overlaps with the 2013 survey article ``Galois representations from pre-image trees: an arboreal survey" of the author. 
\end{abstract}

\maketitle

\section{Introduction}
This article aims to give an overview of the still fairly uncharted area of arboreal Galois representations, which loosely consists of the study of Galois groups acting on preimages under the iteration of self-morphisms of varieties over various fields. Since the preimages of a point form a tree in a natural way, the resulting Galois groups act as tree automorphisms, which is why we term the representations ``arboreal.'' Throughout this survey, we pay particular attention to applications, many of which have appeared only quite recently, and also to open questions and conjectures. This subject belongs generally to the theory of arithmetic dynamics; see [30] for an overview.

We begin by presenting in section 2 four examples of density questions that have nice translations into the setting of arboreal representations, and we give some background for each. In section 3 we give definitions and two general results that will be of much use in the sequel. Then in sections 4-7 we return to each of the questions posed in section 2 and outline results and directions for further work. Finally, in section 8 we present recent work on developing the analogy between arboreal representations and the rich theory of linear $l$-adic representations. Most of the results in this survey appear in [14, 13, 15, 6, 5, 4], excepting section 6 and some of section 8.

We fix some terminology here at the outset: for a map $\phi$ from a set $S$ to itself, denote by $\phi^{n}$ the $n$th iterate of $\phi$, i.e. the $n$-fold composition of $\phi$ with itself. For $s\in S$, call the set $\{\phi(s), \phi^{2}(s), \phi^{3}(s),\dots\}$ the \textit{orbit} of $s$ under $\phi$. Finally, we say that $s\in S$ is \textit{periodic} under $\phi$ if $s$ lies in a cycle, i.e. $\phi^{n}(s)=s$ for some $n\ge 1$. We call $s$ \textit{preperiodic} under $\phi$ if it is not periodic but $\phi^{n}(s)=\phi^{m}(s)$ for some $n, m\ge 1$. Note that if $S$ is finite then every point in $S$ is either periodic or preperiodic.

\section{Four motivating questions}
The problems presented in this section all were essentially open as of 2005, and all turn out to be approachable via the idea of arboreal Galois representations. We merely give background and statements here; detailed treatments follow section 3. As discussed in section 3, all of these questions may be reframed as asking for the density of primes $\p$ such that a given infinite orbit under a particular morphism reduces to a cycle modulo $\p$.

\medskip

\noindent \textbf{Motivating Question 1 (Density of prime divisors of recurrences).} 
The question of which primes appear as divisors of at least one term in a linear recurrence sequence $\{a_{n}\}$ has a considerable history, with many results (see e.g. [2] and [18] for overviews). 

The case of nonlinear recurrences has received scant attention by comparison ([21], [20]). Here we consider this question for a very simple nonlinear recurrence, namely that given by $a_{0}=2$, $a_{n}=(a_{n-1})^{2}+3$. Denote by $P(a_{n})$ the set of primes $p$ such that $p|a_{n}$ for at least one $n\ge 0$. One can show that $P(a_{n})$ is infinite (see [13, Theorem 6.1]). 
Recall that the natural density of a set $S$ of primes in the ring of integers $\mathcal{O}_{K}$ of a number field $K$ is
\begin{equation}
D(S)=\lim_{x\rightarrow\infty}\frac{\#\{\mathfrak{p}\in S:N(\mathfrak{p})\le x\}}{\#\{\mathfrak{p}:N(\mathfrak{p})\le x\}},
\end{equation}
provided that this limit exists. In the present case we take $\mathcal{O}_{K}=\mathbb{Z}$. What is $D(P(a_{n}))$? More generally, what is $D(P(a_{n}))$ if we replace $x^{2}+3$ by $x^{2}+k$ for other $k\in\mathbb{Z}$ and allow $a_{0}$ to be arbitrary?

\medskip

\noindent \textbf{Motivating Question 2 (The order of the reduction of a rational point on an elliptic curve).} 
Consider the elliptic curve $E:y^{2}+y=x^{3}-x$ and the point $\alpha=(0,0)$. What is the density of primes $p$ such that the reduction $\overline{\alpha}\in E(\mathbb{F}_{p})$ has odd order? Computations for small numbers of primes suggest the answer is $1/2$. However, for larger numbers of primes the density appears to hover slightly but noticeably over $1/2$. A very general version of this question has been studied by Pink [23], who showed that the density is positive but gave no method for computing it.

The question can also be rephrased in terms of the elliptic divisibility sequence corresponding to $E$ and $\alpha$, i.e. the sequence $a_{1},a_{2},\dots$ such that $a_{n}$ is the appropriate square root of the denominator of $x([n]\alpha)$ (see [31] for details), where $[n]$ denotes multiplication by $n$. Clearly the order of $\overline{\alpha}\in E(\mathbb{F}_{p})$ is the same as the smallest $n$ such that $p \mid a_{n}$. We refer to the latter quantity as the rank of apparition $r_p$ of $p$, and the above question thus asks for the density of $p$ such that $r_{p}$ is odd. Arithmetic properties of elliptic divisibility sequences reflect properties of the underlying curves, and their study (initiated by Ward [35]) has enjoyed a recent resurgence ([29], [8]).

\medskip

\noindent \textbf{Motivating Question 3 (Density of periodic points for polynomials over $\mathbb{F}_{q}$).}
Fix $q=p^{r}$, and let $f(x)\in\mathbb{F}_{q}[x]$. Write $\overline{\mathbb{F}_{q}}$ instead of the equivalent $\overline{\mathbb{F}_{p}}$, to emphasize that we consider $\mathbb{F}_{q}$ the ground field. Then $f$ acts on $\overline{\mathbb{F}}_{q}$, and the orbit under $f$ of any $\alpha\in\overline{\mathbb{F}}_{q}$ is contained in some finite field and thus is finite. Therefore each $\alpha$ is either periodic or preperiodic. This action on $\overline{\mathbb{F}}_{q}$ has several applications, including the Pollard rho factorization algorithm [24]; see also e.g. [33]. Define the set
\[ \textrm{Per}(f)=\{\alpha\in\overline{\mathbb{F}}_{q}:\alpha \text{ is periodic under } f\}. \]
Given $\mathcal{S}\subseteq\overline{\mathbb{F}}_{q}$ we define its density to be:
\begin{equation}
\delta(\mathcal{S})=\lim_{s\rightarrow 1^{+}}\frac{\sum_{\alpha\in\mathcal{S}}(\deg \alpha)^{-1}N(\alpha)^{-s}}{\sum_{\alpha\in\overline{\mathbb{F}}_{q}}(\deg \alpha)^{-1}N(\alpha)^{-s}}
\end{equation}
where $\deg \alpha=[\mathbb{F}_{q}(\alpha):\mathbb{F}_{q}]$, and $N(\alpha)=q^{\deg \alpha}$. If $f$ is quadratic and $q$ is not a power of 2, what is $\delta(\textrm{Per}(f))$?

\medskip

\noindent \textbf{Motivating Question 4 (The hyperbolic subset of the $p$-adic Mandelbrot set).}
Consider the set
\begin{equation}
\mathcal{H}(\overline{\mathbb{F}}_{p})=\{c\in\overline{\mathbb{F}}_{p}:0 \text{ is periodic under iteration of } x^{2}+c\}.
\end{equation}
The problem here is to find $\delta(\mathcal{H}(\overline{\mathbb{F}}_{p}))$, a project whose interest is greatly enhanced by its connection to the $p$-adic Mandelbrot set. We thus explain this connection, as well as the notation for the left-hand side of (3).

Given a field $K$ and an absolute value $|\cdot|$ on $K$, we define the Mandelbrot set of $K$ to be
\[ M(K)=\{c\in K:0 \text{ has bounded orbit under iteration of } x^{2}+c\}, \]
where we mean bounded with respect to the absolute value. Consider a subset of $M(K)$ that is motivated by the well-known case $K=\mathbb{C}$. Recall that $\phi\in\mathbb{C}(z)$ is said to be hyperbolic if each critical point of $\phi$ tends to an attracting cycle under iteration [17]. We therefore define the hyperbolic Mandelbrot set to be
\[ \mathcal{H}(K)=\{c\in M(K):0 \text{ tends to a formally attracting cycle under iteration of } x^{2}+c\}, \]
where by a formally attracting cycle of $f(x)=x^{2}+c$ we mean that $|f^{\prime}|<1$ at all points of the cycle.
 When the topology on $K$ induced by $|\cdot|$ gives rise to nontrivial geometry, e.g. $K=\mathbb{C}$ and $K=\mathbb{C}_{p}$ a formally attracting cycle is also geometrically attracting. We may decompose $\mathcal{H}(K)$ into a disjoint union of open components $\mathcal{H}(K)^{(i)}$ corresponding to $c$ where 0 tends to a formally attracting $i$-cycle. In the complex case these components form some of the most visible features of $M(\mathbb{C})$. For instance, $\mathcal{H}(\mathbb{C})^{(1)}$ is the main cardioid, and $\mathcal{H}(\mathbb{C})^{(2)}$ is the circle tangent to the cardioid on the real axis. Conjecturally, $\mathcal{H}(\mathbb{C})$ is the interior of $M(\mathbb{C})$; this is the simplest case of the celebrated conjecture that hyperbolic rational maps are open and dense in the space of rational maps of given degree [17]. Moreover, both sets are Lebesgue measurable and the measure of $\mathcal{H}(\mathbb{C})$ exceeds 1.503 while the measure of $M(\mathbb{C})$ is less than 1.562 [9].

We wish to know the size of $\mathcal{H}(K)$ relative to $M(K)$ for $K=\mathbb{C}_{p},$ the smallest complete, algebraically closed extension of $\mathbb{Q}_{p}$. We exclude the case $p=2$. The set $M(\mathbb{C}_{p})$ proves far less topologically interesting than $M(\mathbb{C})$, as $M(\mathbb{C}_{p})$ is just the closed unit disk $\mathcal{O}_{p}$ in $\mathbb{C}_{p}$. However, $\mathcal{H}(\mathbb{C}_{p})$ is not so simple. Letting $\phi:\mathcal{O}_{p}\rightarrow\overline{\mathbb{F}}_{p}$ be the reduction homomorphism, we have $\mathcal{H}(\mathbb{C}_{p})=\phi^{-1}(\mathcal{H}(\mathbb{F}_{p}))$ [14, Proposition 2.2]. Note that since $\overline{\mathbb{F}}_{p}$ admits only the trivial absolute value, we obtain (3). Therefore in a natural sense, $\delta(\mathcal{H}(\overline{\mathbb{F}}_{p}))$ measures the size of $\mathcal{H}(\mathbb{C}_{p})$.

\section{General Arboreal Representations}

In this section we discuss generalities and give two results that are of much use later. Most of the material of this section is taken from [15, Section 2]. In its most general form, an arboreal Galois representation consists of the following data: a variety $V$ defined over a field $K$, a finite, separable morphism $\phi:V\rightarrow V$ of degree $d\ge2$ also defined over $K$, and a point $\alpha\in V(K)$. Let $K^{sep}$ be a separable closure of $K$, and let $U_{n}:=\{\beta\in V(K^{sep})\}$ be the set of $n$th inverse images of $\alpha$ under $\phi$. Define the tree of preimages $T_{\phi,\alpha}$ to be the disjoint union $\sqcup_{n}U_{n}$, which forms a rooted tree with root $\alpha$ when we say that $u, v$ are connected if $\phi(u)=v$ or $\phi(v)=u.$ The Galois group $\text{Gal}(K^{sep}/K)$ acts on $T_{\phi,\alpha}$ preserving the tree structure. This gives a continuous homomorphism
\[ \omega_{\phi,\alpha}:\text{Gal}(K^{sep}/K)\rightarrow \text{Aut}(T_{\phi}(\alpha)) \]
that we call the \textit{arboreal representation} associated to $V$, $\phi$, $\alpha$. We denote the image of this representation by $G_{\phi}(\alpha)$, or simply $G$ when the context is clear. Denote by $K_{n,\phi}(\alpha)$ (or often just $K_{n}$) the field $K(U_{n})$, and let $K_{\infty}$ be the union of the $K(U_{n})$. Similarly, denote by $G_{n,\phi}(\alpha)$ (often just $G_{n}$) the group $\text{Gal}(K_{n}/K)$ so that $G_{\phi}(\alpha) = \varprojlim G_{n}$. We also denote by $\phi^{n}$ the $n$th iterate of $\phi$, that is, the $n$-fold composition of $\phi$ with itself. We remark that in this work we always stipulate that $\phi$ have degree at least two.

Generally it is quite difficult to determine $G_{\phi}(\alpha)$ for specific choices of $K$, $V$, $\phi$, and $\alpha$ (see discussion in section 8.1). However, when $K$ is a global field and one can gain some information about $G_{\phi}(\alpha)$, it is possible to parlay that into density information about behavior of the reduced orbit of $\alpha$ modulo primes of the ring of integers $\mathcal{O}_K$ of $K$. The aim of the remainder of this section is to establish this fact and give a useful criterion for showing the associated density is zero. This latter criterion can be found in Theorem 2, and makes essential use of the theory of stochastic processes. It is the basis for many of the applications described in subsequent sections.

Before stating the main results of this section, we require some background and definitions. To motivate the first definition, suppose for a moment that $K=\mathbb{Q}$, and let $\phi\in\mathbb{Q}(x)$ be a rational function. As discussed below, for all but finitely many primes $p$, one may reduce modulo $p$ and obtain a well-behaved map $\overline{\phi}$ from $\mathbb{Z}/p\mathbb{Z}$ to itself. If $\overline{\alpha}\in\mathbb{Z}/p\mathbb{Z}$ is periodic under $\overline{\phi}$, it has at least one preimage in $\mathbb{Z}/p\mathbb{Z}$ under $\overline{\phi}^{n}$ for every $n$. Now points in $\mathbb{Z}/p\mathbb{Z}$ are precisely those fixed by the Frobenius conjugacy class at $p$ in $\text{Gal}(\overline{\mathbb{Q}}/\mathbb{Q})$. As mentioned at the beginning of section 2, all four motivating questions are essentially asking to count primes $p$ modulo which $\alpha$ is periodic. Thus to do this it is natural to look for elements of $T_{\phi}(\alpha)$ fixed by the Frobenius class at various primes. The Chebotarev density theorem says that the proportion of $p$ whose Frobenius conjugacy class is a given class $C$ (in a finite Galois group $G$) is $\#C/\#G$. Thus define:
\begin{equation}
\mathcal{F}(G_{\phi}(\alpha)):=\lim_{n\rightarrow\infty}1/\#G_{n}\cdot\#\{g\in G_{n}:g \text{ fixes at least one point in } U_{n}\}. 
\end{equation}

Note that one may also define $\mathcal{F}(G_{\phi}(\alpha))$ as the Haar measure of the set of elements acting on the ends of T with at least one fixed point. We frequently use the following notion of density, which essentially generalizes both and D. If S is a set of primes in K, put
\begin{equation}
\Delta(S)=\lim_{s\rightarrow1^{+}}\frac{\sum_{\mathfrak{p}\in S}N(\mathfrak{p})^{-s}}{\sum_{\mathfrak{p}}N(\mathfrak{p})^{-s}}, 
\end{equation}
where $N(\mathfrak{p})$ denotes the norm of $\p$. 

Suppose that K is a global field with ring of integers $\mathcal{O}$. As described in [15, Section 2] (see also [16, pp. 107-108] for a detailed discussion), one may consider V as a scheme over Spec $\mathcal{O}$. For each prime $\mathfrak{p}\subset\mathcal{O}$ denote the residue field corresponding to $\p$ by $k_{\mathfrak{p}}$, the fiber of $V$ over $\p$ by $V_{\mathfrak{p}}$ and the $k_{\mathfrak{p}}$ points of $V_{\mathfrak{p}}$ by $V(k_{\mathfrak{p}}).$ Then given $\alpha\in V(K),$ for all but a finite number of primes $\p$ there is a well-defined reduction $\overline{\alpha}\in V(k_{\mathfrak{p}})$ and a well-defined reduced morphism $\overline{\phi}:V_{\mathfrak{p}}\rightarrow V_{\mathfrak{p}}$ with $\text{deg}~\overline{\phi} = \text{deg}~\phi.$

We now formally state the connection between $\mathcal{F}(G_{\phi}(\alpha))$ and densities:

\begin{theorem}  Assuming the notation above, we have
$$\mathcal{F}(G_{\phi}(\alpha))\ge\Delta(\{\mathfrak{p}\subset\mathcal{O}:\overline{\alpha}\in V(k_{\mathfrak{p}}) \text{ is periodic under } \overline{\phi}\}).$$
If in addition $K_{\infty}/K$ is finitely ramified, then we obtain equality. In the number field case, $\Delta$ may be replaced by D. 
\end{theorem}

For a proof, see [15, Section 2].

Clearly Theorem 1 is particularly useful for showing that certain densities are zero. Several cases of interest to us (for instance, Motivating Questions 1, 3, and 4) occur when $V = \mathbb{A}^1$ and $\phi$ is a quadratic polynomial. In this case, the group $H_n := \Gal(K_n/K_{n-1})$ is an elementary abelian 2-group of rank at most $2^{n-1}$ (see [14, Proposition-Definition 5.1]), and when the rank equals $2^{n-1}$ we call $H_n$ \textit{maximal}. The following theorem, which makes essential use of the theory of stochastic processes, gives a powerful tool for showing that certain densities are zero. 

\begin{theorem} \label{zerocrit}
Suppose that $V = \mathbb{A}^1$, $\phi : V \to V$ is a quadratic polynomial, and $K$ does not have characteristic 2.  Suppose also that
$H_n$ is maximal for infinitely many $n$ and $\Disc \phi^n$ is not a square in $K$ for all $n$.  Then $\f(G_{\phi}(\alpha)) = 0$.
\end{theorem}
\begin{remark} The hypothesis that $\Disc \phi^n$ is not a square is necessary only in that it ensures $\phi^n$ is irreducible over $K$ (and thus $G_n$ acts transitively on $U_n$) and $G_n$ is not alternating.  See [14, Lemma 4.11] for a proof that $\Disc \phi^n$ not a square implies $\phi^n$ irreducible.
\end{remark}

\begin{proof}(sketch):
Write $G$ instead of $G_{\phi}(\alpha)$, and $G_n$ instead of $G_{n, \phi}(\alpha)$.  Following [14], denote by $GP(\phi)$ the stochastic process defined as follows.  Take $\P$ to be the Haar
measure on $G$ with $\P(G) = 1$, and $\pi_n$ to be the natural projection $G \rightarrow G_n$.  We define random variables on $G$ by setting $X_n(g)$ to be the number of roots of $\phi^n$ fixed by $\pi_n(g)$.  It follows that
\begin{equation} \label{galprocprop}
\P(X_n > 0) = \frac{1}{\#G_n} \cdot \# \{g \in G_n : \text{$g$ fixes at
least one root of $f_n$}\}.
\end{equation}
The assumption that $\Disc \phi^n$ is never a square implies that $GP(\phi)$ is a martingale (see proof of [14, Theorem 1.2]), i.e. that for all $n \geq 2$ and any $t_i \in \mathbb{Z}$,
$$E(X_n \mid X_{1} = t_{1}, X_2 = t_2, \ldots, X_{n-1} = t_{n-1}) = t_{n-1},$$
provided $\P(X_{1} = t_{1}, X_2 = t_2, \ldots, X_{n-1} = t_{n-1}) > 0$.  By a standard result from the theory of stochastic processes (see e.g. [11, Section 12.3]), this implies that $GP(\phi)$ converges, which in our case means
\begin{equation} \label{const}
\P(\{g \in G : \text{$X_1(g), X_2(g), \ldots$ is eventually constant}\}) = 1.
\end{equation}
From [14, Lemma 5.3] we have that $H_n$ maximal implies that for any $m < n$ and
$u > 0$,
\[
\P(X_n = u \mid X_{m} = u, \ldots, X_{n-1} = u) \leq \frac{1}{2}.
\]
From \eqref{const} and our assumption that $H_n$ is maximal for infinitely many $n$, we get
$\inflim{n} \P(X_n(g) = 0) = 1$ (see [14, Proof of Theorem 1.3]).  This proves the theorem.
\end{proof}

\smallskip

\section{Motivating Question 1: the density of prime divisors of certain non-linear recurrences} \label{recurrences}

Motivating Question 1 asked about the prime divisors of the sequence $a_0 = 2$, $a_n = a_{n-1}^2 + 3$ for $k \in \Z$.  Note that if $\phi = x^2 + 3$, then we have $a_n = \phi^n(2)$.
In this section we give a slight alteration of Theorem 1 and then apply Theorem \ref{zerocrit} to the case where $K = \Q$, $V = \mathbb{A}^1$, $\alpha = 0$, and $\phi = x^2 + 3$ (indeed we look at $\phi = x^2 + k$ for any $k \in \Z$ with
$-k$ not a square).  These assignations for $K, V$, and $\alpha$ are in force throughout this section; thus $H_n$, $U_n$, etc. are understood in this context.  We allow $\phi$ to vary somewhat, as we wish to describe similar results for $\phi$ belonging to three other families of quadratic polynomials, and to state some conjectures for $\phi$ a general quadratic polynomial.  Most of the results in this section appear in one form or another in the preprint [13].

We first give the following result that is a variant of Theorem 1:
\begin{theorem} \label{polymaingen}
Let $\phi \in \Z[x]$ be a polynomial with $\phi^n$ separable for all $n$.  Let $a_n = \phi^n(a_0)$ with $a_0 \in \Z$.  Then $D(P(a_n)) \leq \f(G_{\phi}(0))$.
\end{theorem}

\begin{proof}(sketch) The idea is that if $\phi^N(x) = 0 \bmod{p}$ has no solution in $\Z$, then
$p \nmid a_n$ for $n > N$.  Since only finitely many $p$ divide some $a_n$ for $n \leq N$, the density of $p$ with $\phi^N(x) = 0 \bmod{p}$ having a solution in $\Z$ must exceed $D(P(a_n))$.  But $\phi^N(x) = 0 \bmod{p}$ having a solution in $\Z$ is equivalent to $\phi^N(x) \in \Z/p\Z[x]$ having a linear factor, i.e. to the Frobenius conjugacy class at $p$ fixing a root of $\phi^N(x)$.  By the Chebotarev Density Theorem, the density of $p$ with this property is exactly the proportion of
$\sigma \in G_n$ fixing a root of of $\phi^N(x)$, i.e. an element of $U_n$.  Letting $N$ go to infinity gives the result.  \end{proof}

We now wish to apply Theorem 2. We thus take $\phi=ax^{2}+bx+c\in\mathbb{Z}[x]$ irreducible, and let $\gamma=-b/2a$ be the critical point of $\phi$. It turns out that both hypotheses of Theorem 2 boil down to arithmetic properties of the critical orbit $(\phi^{n}(\gamma):n\ge1)$ of $\phi$. This provides an analogy to results in real and complex dynamics, where analytic properties of the critical orbit have been shown to determine many global dynamical properties of a quadratic polynomial. 

To handle the hypothesis of Theorem 2 that $\textrm{Disc}$ $\phi^{n}$ never be a square, one can show (see [13, Lemma 2.8]) that
\begin{equation} \label{eq:8} 
\text{Disc } \phi^{n}=\pm a^{2^{2n-1}-1}2^{2^{n}}(\text{Disc } \phi^{n-1})^{2}\phi^{n}(\gamma). \text{ [cite: 7]}
\end{equation}

In particular, for $n\ge2$ we have that $\text{Disc } \phi^{n}$ is a square if and only if $\phi^{n}(\gamma)$ is a square. 

The second hypothesis of Theorem 2 is that $H_{n}$ is maximal for infinitely many $n$. 
To give a criterion for this, we use the following special case of [13, Theorem 3.3]: 

\begin{theorem}
Let $\phi=ax^{2}+bx+c\in\mathbb{Z}[x]$ and let $\gamma$ be the critical point of $\phi$. Suppose that $\phi^{n}$ is irreducible over K for all $n\ge1$. Denote by $v_{p}$ the p-adic valuation. If $n\ge2$ and there exists p with $v_{p}(\phi^{n}(\gamma))$ odd, $v_{p}(\phi^{m}(\gamma))=0$ for all $1\le m\le n-1$, and $v_{p}(2a)=0$, then $H_{n}$ is maximal.
\end{theorem}

To use Theorem 4, we require a method for showing that $\phi^{n}$ is irreducible for all $n$. This also hinges on the critical orbit of $\phi$. Indeed, from [13, Theorem 4.3] it follows that if $\phi$ is irreducible and $\phi^{n}(\gamma)$ is not a square for all $n\ge2,$ then all iterates of $\phi$ are irreducible. 

There are certain families of polynomials where the arithmetic of the critical orbit is relatively easy to understand. For instance, let us return to the case of Motivating Question 1. Taking $\phi=x^{2}+k$ we have $\gamma=0$. Put $c_{n}=\phi^{n}(0)$ and consider the critical orbit $(c_{n}:n\ge1)$ of $\phi$. One can show this sequence is a \textit{rigid divisibility sequence}, namely that $\gcd(c_{m},c_{n})=c_{\gcd(m,n)}$ for all $n, m$, and for any prime $p$ and integer $n,$ $v_{p}(c_{n})=e>0$ implies $v_{p}(c_{mn})=e$ for all $m\ge1$ The first property follows from the fact that $c_{n+k}=\phi^{n+k}(0)=\phi^{k}(c_{n})\equiv c_{k}$ mod $c_{n}$, while the second follows from the fact that $\phi$ has no linear term. See [13, Lemma 5.3] for proofs.

This strong property of the critical orbit of $\phi$ allows us to answer Motivating Question 1:

\begin{theorem}
Let $\phi = x^2 + k \in \Z[x]$ with $-k$ not a square, and consider the sequence
$a_n = \phi^n(a_0)$ for some $a_0 \in \Z$.  We have $D(P(a_n)) = 0$.
\end{theorem}

\begin{proof} As above, let $(c_{n}:n\ge1)$ be the critical orbit of $\phi$. First, note that an elementary argument gives that $(c_{n}:n\ge2)$ cannot contain any squares. Since $-k$ is not a square, the critical orbit contains no squares. Thus from (8), $\text{Disc }$ $\phi^{n}$ is not a square for all $n \geq 1$ and from the remark following Theorem 4, $\phi^n$ is irreducible for all $n$. 

We now need only apply Theorem 4, and then by Theorem 2 the proof is complete. The irreducibility of $\phi^n$ is established. Since the $c_n$ form a rigid divisibility sequence, we have that for any prime $\ell$, $c_1 \mid c_\ell$ and $c_\ell/c_1$ is relatively prime to $c_1$. By the previous paragraph we know that $c_\ell/c_1$ is relatively prime to $c_2, \ldots, c_{n-1}$. Thus from Theorem 4 we have that $H_\ell$ is maximal if $c_\ell/c_1$ is not a square or twice a square. This must happen for infinitely many (indeed all but finitely many) $\ell$, as otherwise the equation $y^2 = \phi^2(x)$ (resp. $2y^2 = \phi^2(x)$) would have infinitely many integral solutions, contradicting Siegel's theorem [12, p. 353].
\end{proof}

\begin{remark}
Note that Theorem 5 follows in many cases (notably $k = 1$, which was the subject of a question of Odoni [22] from the result of Stoll [32], which gives infinitely many values of $k$ for which $H_n$ is maximal for all $n$.  This maximality for all $n$ means one does not need to make use of Theorem 2, and may instead employ more direct methods, e.g. those in [14, Section 5] (see also [21]).  However, Stoll's result does not cover the case mentioned in Motivating Question 1, as indeed one can show that for $\phi = x^2 + 3$, $H_3$ is not maximal (it has order $2^3$ instead of the maximal $2^4$).
\end{remark}

In a manner similar to the proof of Theorem 5, one can obtain zero-density results for sequences whose critical orbits obey strong arithmetical properties. The following is a special case of [13, Theorem 1.1].

\begin{theorem}
Let $\phi \in \mathbb{Z}[x]$ be monic and quadratic, and let $\gamma$ be the critical point of $\phi$. Suppose that $\phi^n$ is irreducible for all $n \ge 0$ and the set $\{\phi^n(\gamma)\}$ is infinite. If either
\begin{enumerate}
    \item The set $\{\phi^n(0) : n = 1, 2, \dots\}$ is finite and does not contain 0, or
    \item the sequence $(\phi^n(\gamma) : n = 1, 2, \dots)$ is a rigid divisibility sequence,
\end{enumerate}
then $\mathcal{F}(G_\phi(0)) = 0$.
\end{theorem}

\begin{remark}
Case (1) of Theorem 6 implies that elements of $(\phi^n(\gamma))$ are nearly pairwise relatively prime, which is quite strong. Indeed in this case $G_\phi(0)$ has finite index in $Aut(T_\phi(\alpha))$ (see also discussion following Question 21), so in other words $H_n$ is maximal for all but finitely many $n$.
\end{remark}

\begin{remark}
Theorem 1.1 of [13] deals with translated iterates, i.e. Galois groups of polynomials of the form $\psi \circ \phi^n$, where $\psi, \phi \in \mathbb{Z}[x]$. In order to do this, in [13] I develop generalizations of Theorems 3 and 4, as well as irreducibility results, that cover translated polynomial iterates. This greater generality aids in handling cases where some iterate of $\phi$ is reducible, such as $\phi = x^2 + k$ for $-k$ a square (as long as $k \neq -1$, a case that presents peculiar difficulties; see the discussion following [13, Theorem 5.2]). 
\end{remark}

One can apply a generalization of Theorem 6 to several families of polynomials, obtaining [13, Theorem 1.2]:

\begin{theorem} Suppose $a_{n}=\phi(a_{n-1})$ with $a_{0}\in\mathbb{Z}$ arbitrary, and that one of the following holds:
\begin{enumerate}
    \item $\phi=x^{2}-kx+k$ for some $k\in\mathbb{Z}$
    \item $\phi=x^{2}+kx-1$ for some $k\in\mathbb{Z}\setminus\{0,2\}$
    \item $\phi=x^{2}+k$ for some $k\in\mathbb{Z}\setminus\{-1\}$
    \item $\phi=x^{2}-2kx+k$ for some $k\in\mathbb{Z}\setminus\{\pm1\}$
\end{enumerate}
Then $D(P(a_{n}))=0$ for all $a_{0}\in\mathbb{Z}$.
\end{theorem}

To summarize, we should be able to establish a density zero result for any sequence of the form $a_{n}=\phi(a_{n-1})$ where $\phi\in\mathbb{Z}[x]$ is quadratic, provided the critical orbit satisfies the condition of Theorem 4. (The conditions that $\text{Disc}$ $\phi^n$ not a square and $\phi^{n}$ irreducible for all $n$ hold generically, as detailed in see [13, Section 4], and are easy to verify in any specific case.) Computations and heuristics suggest that this should indeed be the case for essentially any $\phi$, leading us to the following conjecture, which is also Conjecture 5.7 of [13]:

\begin{conjecture} Let $\phi\in\mathbb{Z}[x]$ be quadratic, and suppose that $\phi$ has infinite critical orbit and all iterates irreducible. Then the sequence $a_{n}=\phi(a_{n-1})$ with $a_{0}\in\mathbb{Z}$ arbitrary satisfies $D(P(a_{n}))=0$.
\end{conjecture}

Note that the property of having infinite critical orbit is clearly generic, as it only fails for conjugates of $x^{2}$, $x^{2}-1$, and $x^{2}-2$.

\section{Motivating Question 2: reduction statistics for endomorphisms of abelian algebraic groups} \label{alggp}

Recall that Motivating Question 2 asked about the elliptic curve $E:y^{2}+y=x^{3}-x$ and the point $\alpha=(0,0)$, and requested the density of $p$ such that the order of $\overline{\alpha}\in E(\mathbb{F}_{p})$ is odd.
To answer this question, we let $V$ be an abelian algebraic group $A$ and take $\phi$ to be the multiplication-by-$\ell$ endomorphism on $E$. Unlike sections 4, 6, and 7, in this section we allow $K$ and $\alpha$ to vary, although $V$ will always be some abelian algebraic group and $\phi$ will always be an endomorphism of $V$.

Our first proposition translates the question about orders of reductions into a question about arboreal representations. Let $K$ be a global field over which $V, \phi$, and $\alpha$ are defined, and we consider prime ideals $\mathfrak{p}$ of the subring $\mathcal{O}_K\subset K$. Also, $k_{\mathfrak{p}}$ denotes the residue field modulo $\mathfrak{p}$ and $V(k_{\mathfrak{p}})$ the $k_{\mathfrak{p}}$-points of $V$.

\begin{proposition} \label{interp}
If $V = A$ is a torus or an abelian variety and $\phi =
[\ell]$ is the multiplication-by-$\ell$ map for a prime $\ell$, then $\f(G_{\phi}(\alpha))$ is the density of $\mathfrak{p}$ such that
the order of $\overline{\alpha} \in A(k_{\p})$ is not divisible
by $\ell$.
\end{proposition}
\begin{remark} The hypothesis that $\ell$ be prime is not necessary in Proposition \ref{interp}, but is required for subsequent computations.
\end{remark}

\begin{proof} First note it is a standard result that $K_{\infty}/K$ is a finitely ramified extension (see e.g. [12, p. 263] for the abelian variety case).
  Thus by Theorem 1 we have $\mathcal{F}(G_{\phi}(\alpha)) = D(\{\p \subset \mathcal{O} : \text{$\overline{\alpha} \in
    V(k_{\p})$ is periodic under $\ell$}\}).$ Denote the order of
  $\overline{\alpha} \in A(k_{\p})$ by $m$.  We show that $\ell \mid
  m$ if and only if $\overline{\alpha}$ is periodic under $\ell$.  If
  $\ell \nmid m$ then $\ell \in (\Z/m\Z)^{\times}$, whence $\ell^n
  \equiv 1 \bmod{m}$ for some $n$.  Thus $[\ell^n] \overline{\alpha} =
  \overline{\alpha}$, whence $\overline{\alpha}$ is periodic under
  $\ell$.  Conversely, if $[\ell^n] \overline{\alpha} =
  \overline{\alpha}$ for some $n$, then $[\ell^n - 1]
  \overline{\alpha} = \overline{0}$, whence $\ell$ cannot divide the
  order of $\overline{\alpha}$.
\end{proof}

For the remainder of this section, we restrict to the case where the hypotheses of Proposition \ref{interp} are satisfied, and we denote $G_{\phi}(\alpha)$ by $G_{\ell}(\alpha)$.  Proposition \ref{interp} reduces our problem to one of computing $\f(G_{\ell}(\alpha))$.  However
for this problem, in contrast to the results of sections 4 and 7, the groups $G_{\ell}(\alpha)$ are very small subgroups of $\Aut(T_{\ell}(\alpha))$, and the resulting densities will be positive.  The smallness of $G_{\ell}(\alpha)$ comes from the fact that its elements must respect the underlying group structure on the algebraic group $A$.  Indeed, Let $A[\ell^{n}] = \{ \gamma \in K^{\sep} : \ell^{n}(\gamma) = 0 \}$ be the $\ell^n$-torsion points, and $T_{\ell}(A) := \varprojlim A[\ell^{n}]$ be the Tate module of $A$ (note that in the notation of section 3, $T_{\ell}(A)$ is the same as $T_{\ell}(O)$, where $O \in A$ is the identity).  Choosing elements $\beta_1 \in U_1, \beta_2 \in U_2, \ldots$ such that $\ell \beta_n = \beta_{n-1}$, the homomorphisms
$$\Gal(K_{n}/K) \to A[\ell^{n}] \rtimes \Aut(A[\ell^n])$$
given by $\sigma \mapsto (\sigma(\beta_{n}) - \beta_{n}, \sigma|_{A[\ell^{n}]})$ patch together to give a homomorphism
$$G_{\ell}(\alpha) = \Gal(K_{\infty}/K) \to T_\ell(A) \rtimes \Aut(T_{\ell}(A)).$$  It is straightforward to show that this homomorphism is injective [15, Proposition 7].  In the cases of interest to us,
$A[\ell^{n}] \cong (\Z/\ell\Z)^e$ for some $e$, meaning $G_{\ell}(\alpha)$ is finitely generated as a pro-$\ell$ group.  Thus it is a very small subgroup of $\Aut(T)$.  To study $G_{\ell}(\alpha)$, we break it into pieces, as summarized by the following diagram:

\[
\xymatrix{
{\scriptstyle 1} \ar[r] & {\scriptstyle \Gal(K_{\infty}/K(A[\ell^{\infty}]))} \ar[r]
\ar[d]^\kappa &
{\scriptstyle \Gal(K_{\infty}/K)} \ar[r] \ar[d]^\omega &
{\scriptstyle \Gal(K(A[\ell^{\infty}])/K)} \ar[r] \ar[d]^\rho & {\scriptstyle 1}\\
{\scriptstyle 1} \ar[r] & {\scriptstyle T_{\ell(A)}} \ar[r] &
{\scriptstyle T_{\ell(A)} \rtimes \Aut(T_{\ell}(A))}
  \ar[r] & {\scriptstyle \Aut(T_{\ell}(A))} \ar[r] & {\scriptstyle 1}\\
}
\]

The rows are exact, the maps on the top row being the natural
ones. The nontrivial maps on the bottom row are inclusion into
the first factor, and projection onto the second factor,
respectively.  The vertical arrows are all injections, with $\rho$ being the usual
$\ell$-adic representation and $\kappa$ corresponding to a kind of $\ell$-adic Kummer extension.
The image of $\rho$ has been extensively studied (e.g. [26]), and in the case of elliptic curves one may verify its surjectivity easily in any given case (matters are not so simple for general abelian varieties, but there are still criteria available; see [34], [7], and discussion in [15, Section 6]).  The image of the map $\kappa$ has been studied for the $\bmod{\;\ell}$ version of the above diagram [25], and those results have been extended in [3, Theorem 2, p.40], which states that if $A$ is an abelian
variety or the product of an abelian variety by a torus, then the image of $\kappa$ is the full Tate module for all but finitely many $\ell$ and has open image for all $\ell$ (see also
[23, Theorem 2.8] and [10, Proposition 2.10] for the
latter statement).  However, for our purposes, we need to know precisely what the image of $\kappa$ is for specific $\ell$.  We thus prove the following specific criterion in the case $A = E$ is an elliptic curve 
defined over a number field $K$ [15, Theorem 20]:
\begin{theorem}
\label{surj2}
Let $E$ be an elliptic curve, and suppose that the $\ell$-adic
representation $\rho$ associated to $E$ is surjective.  Then the Kummer map
$\kappa : \Gal (\overline{K}/ K(E[\ell^{\infty}])) \to \Z_{\ell}^2$ is surjective
if and only if $\alpha \not\in \ell E(K)$.
\end{theorem}
The proof is a somewhat lengthy computation; we refer the reader to [15].
Combining Theorem \ref{surj2} with known surjectivity criteria for $\rho$ in the elliptic curve case, we obtain:
\begin{theorem} \label{overallsurj2}
Let $A = E$ be an elliptic cirve.  Then $G_{\ell}(\alpha)$ is all of $\Z_{\ell}^2 \rtimes \GL_2(\Z_{\ell})$
if and only if all the following conditions hold:
\begin{enumerate}
\item $\alpha \not\in \ell E(K)$.
\item $K$ is linearly disjoint from $\Q(\zeta_{\ell^{n}})$
  for all $n$.
\item $\Gal(K(E[\ell])/K) \cong \GL_{2}(\Z/\ell \Z)$.
\item If $\ell = 2$, then $K(E[2])$ is linearly disjoint
  from $\Q(\sqrt{2}, i)$.
\item If $\ell = 3$, then the 9-torsion polynomial is
  irreducible over $K(\zeta_{9})$.
\end{enumerate}
\end{theorem}
We prove similar results in the cases where $A$ is an elliptic curve with complex multiplication [15, Corollary 29] one-dimensional torus [15, Corollary 13] or a higher-dimensional abelian variety [15, Corollary 37].

In order now to find densities in specific cases, we need a method for computing $\f(G_{\ell}(\alpha))$.
The following is a special case of [15, Theorem 10]:
\begin{theorem}
\label{matrixprop} Let $E, \ell,$ and $\alpha$ be such that
$G_{\ell}(\alpha)$ is all of $\Z_{\ell}^2 \rtimes \GL_2(\Z_{\ell})$.
Let $\mu$ be the natural Haar measure on the pro-$\ell$ group $\GL_2(\Z_{\ell})$, normalized
so that $\mu(\GL_2(\Z_{\ell})) = 1.$
Then
\begin{equation}
\label{integ}
  \f(G_{\ell}(\alpha)) = \int_{\GL_2(\Z_{\ell})} \ell^{-\ord_{\ell}(\det(M-I))} \, d\mu.
\end{equation}
\end{theorem}

Finally, in [15, Theorem 24], a lengthy computation shows the integral in \eqref{integ} is equal to
\begin{equation} \label{formula}
\frac{\ell^5 - \ell^4 - \ell^3 + \ell + 1}{\ell^5 - \ell^3 - \ell^2 + 1}.
\end{equation}
In particular, when $\ell = 2$, we have $\f(G_\ell(\alpha)) = 11/21$, provided that $E$ and $\alpha$ satisfy the conditions of Theorem \ref{overallsurj2}.  Therefore to answer Motivating Question 2, all that remains is to show that $E: y^2 + y = x^3 + x$ and $\alpha = 0$ satisfy the aforementioned conditions; this is done in [15, Example 23].  Hence the density of $p$ with $\overline{\alpha} \in E(\Fp)$ having odd order is 11/21.

In [15], we also obtain formulas similar to \eqref{formula} in the case where $\kappa$ and $\rho$ are surjective and $A$ is a one-dimensional torus or an elliptic curve with complex multiplication.  The answers are quite different.  Indeed, if we had chosen a typical pair $(E, \alpha)$ where $2$ splits in the CM ring of $E$, the answer to Motivating Question 2 would have been only 2/9 (see [15, Example 33]) However, in the case where $A$ is a higher-dimensional abelian variety, the corresponding question appears difficult.  Even the case of $\dim A = 2$ is unresolved:
\begin{question}
Let $A$ be an abelian surface with $G_{\ell}(\alpha)$ as large as possible, i.e. $\Z_{\ell}^4 \rtimes \GSp_4(\Z_{\ell})$.  What is
$\f(G_{\ell}(\alpha))$?
\end{question}

Another open question arises if we fix a prime, say $\ell = 2$, and assume that $A_d$ is an abelian variety of dimension $d$ and $\alpha$ is a point such that $G_{\ell}(\alpha)$ is as large as possible, namely $\Z_{\ell}^{2d} \rtimes \GSp_{2d}(\Z_{\ell})$.  We may then ask for
$$\inflim{d} \f(G_{\ell}(\alpha)).$$
That this limit exists is shown by J. Achter in [15, Appendix A]

\section{Motivating Question 3: the density of periodic points under polynomials in $\Fqbar$} \label{finfld}

Recall that $f(x) \in \mathbb{F}_{q}[x]$ acts on $\overline{\mathbb{F}}_{q}$, and any $a \in \overline{\mathbb{F}}_{q}$ is either periodic or preperiodic under $f$. We define $\textrm{Per}(f) \subset \overline{\mathbb{F}}_{q}$ to be the set of periodic $a$, and ask for the density $\delta(\textrm{Per}(f))$.
This is likely to have applications to the many integer factorization algorithms, including Pollard's rho algorithm, that rely on iteration of polynomials defined over finite fields. Indeed, cycles are expected to be relatively short, so if $\delta(\Per(f)) > 0$ then with this probability the rho algorithm fails within a short time.

To translate this question into one involving arboreal representations, we take $K = \mathbb{F}_{q}(t)$ (thus $\mathcal{O} = \mathbb{F}_{q}[t]$), $V = \mathbb{A}^{1}$, and $\alpha = t$. These assignations will be fixed throughout this section. As in section 4, we take $\phi = f \in \mathbb{F}_{q}[x]$ and allow this to vary somewhat.
We begin with a lemma that connects the two notions of density $\Delta$ (see (5)) and $\delta$ (see (2)).

\begin{lemma}
Suppose that $\mathcal{S} \subseteq \mathbb{F}_{q}$ is invariant under the action of $\text{Gal}(\overline{\mathbb{F}}_{q}/\mathbb{F}_{q})$, and let $T$ be the set of primes of $\mathcal{O} = \mathbb{F}_{q}[t]$ given by $\{(\pi_{s}) : s \in \mathcal{S}\}$, where $\pi_{s}$ is the minimal polynomial of $s$. Suppose also that $\Delta(T)$ exists. Then $\delta(\mathcal{S})$ exists and equals $\Delta(T)$.
\end{lemma}

\begin{proof}
A straightforward adaptation of [14, Theorem 3.3]. All that is required to make that argument work is $\deg(s) = \deg \pi_{s}$ and $N(s) = N(\pi_{s})$ both of which hold here.
\end{proof}

We now reduce our question to one involving arboreal representations.

\begin{proposition}
With $K, V, \alpha, f$ as above, and $\textrm{char}(K)$ not dividing $\deg f$, we have 
$$\mathcal{F}(G_{f}(t)) = \delta(\Per(f)).$$
\end{proposition}

\begin{proof}[Proof (sketch)]
It follows from the finiteness of $\mathbb{F}_{q}$ that the extension $K_{\infty}/K$ is finitely ramified; indeed it is ramified only at primes of the form $(b-t)$, where $b = f^{n}(c)$ for some $n$ and some critical point $c$ of $f$ (c.f. [1, Theorem 1.1]). From Theorem 1, we now have
\begin{equation}
\mathcal{F}(G_{f}(t)) = \Delta(\{\mathfrak{p} \subset \mathcal{O} : t \bmod \mathfrak{p} \in \mathcal{O}/\mathfrak{p} \text{ is periodic under } \overline{f}\}).
\end{equation}
Now since $f$ is defined over the constant field of $K$, we have $\overline{f} = f$. Note also that $s \in \Per(f)$ if and only if $(t \bmod \pi_{s}) \in \mathcal{O}/\pi_{s}$ is periodic under $f$, where $\pi_{s}$ is the minimal polynomial over $\mathbb{F}_{q}$ of $s$. Thus $\Per(f)$ is Galois-invariant, and is equal to the set on the right-hand side of (11). The proposition now follows from Lemma 14.
\end{proof}

We therefore need to understand the group $G_{f}(t)$. This group has two components, namely the normal subgroup $\text{Gal}(K_{\infty}/\overline{\mathbb{F}}_{q}(t))$ and its procyclic quotient $\text{Gal}((K_{\infty} \cap \overline{\mathbb{F}}_{q})/\mathbb{F}_{q})$. The first group is similar in many ways to the iterated monodromy group of a critically finite polynomial $F \in \mathbb{C}[z]$, and the latter are generated by a finite automaton [19, Chapter 5]. It thus seems reasonable to suspect that $\text{Gal}(K_{\infty}/\overline{\mathbb{F}}_{q}(t))$ is also generated by finite automata. Moreover, finding $\text{Gal}((K_{\infty} \cap \overline{\mathbb{F}}_{q})/\mathbb{F}_{q})$ is the same as finding the action of Frobenius on $G_{f}(t)$, which should be possible to do explicitly. Indeed, the entire action of $G_{f}(t)$ on $\text{Aut}(T)$ may well be given by a finite automaton.
This information should allow one to compute $\mathcal{F}(G_{f}(t))$, at least in the case where $f$ is quadratic. In this case, the extension $K_{\infty}$ of $K$ is finitely ramified, implying $H_{n}$ can be maximal for only finitely many $n$, whence Theorem 2 will never apply. However, the techniques involved in the proof of Theorem 2 may still be brought to bear, leading us to make the following conjecture:

\begin{conjecture}
Let $f \in \mathbb{F}_{q}[x]$ be a monic quadratic polynomial that is not conjugate to $x^{2}$ or $x^{2}-2$. Then $\delta(\Per(f)) = 0$.
\end{conjecture}

We also propose another conjecture, which would make the computation of $\mathcal{F}(G_{f}(t))$ far easier. We state it here for the case of $T$ the complete infinite rooted binary tree, though a similar conjecture should hold for $T$ of higher valency. For a subgroup $G \le \text{Aut}(T)$, define the Hausdorff dimension $h(G)$ to be
\begin{equation*}
h(G) = \lim_{n\rightarrow\infty} \frac{\log_{2}(\#G_{n})}{\log_{2}(\#\text{Aut}(T_{n}))},
\end{equation*}
where $T_{n}$ is the complete binary rooted tree of height $n$ and $G_{n}$ is the image of $G$ in $\text{Aut}(T_{n})$.

\begin{conjecture}
Let $G$ be a level-transitive subgroup of $\text{Aut}(T)$ with $h(G) > 0$. Then $\mathcal{F}(G) = 0$.
\end{conjecture}

It should not be difficult to establish that for monic, quadratic $f$, $h(G_{f}(t)) > 0$ unless $f$ is conjugate to $x^{2}$ or $x^{2}-2$ thereby showing that Conjecture 17 implies Conjecture 16.

\section{Motivating Question 4: the density of the hyperbolic subset of the $p$-adic Mandelbrot set} \label{padic}

The results of this section can be found in [14]. Recall that we wish to find $\delta(\mathcal{H}(\overline{\mathbb{F}}_{p}))$, where
\[ \mathcal{H}(\overline{\mathbb{F}}_{p}) = \{c \in \overline{\mathbb{F}}_{p} : 0 \text{ is periodic under iteration of } x^2 + c\}. \]
Our strategy is to define sets that give successively better ``approximations'' of $\mathcal{H}(\overline{\mathbb{F}}_{p})$, and show that their density approaches 0. To show this latter statement, we make use of the specific arboreal representation with $K = \mathbb{F}_p(t)$, $V = \mathbb{A}^1$, $\alpha = 0$, and $\phi = x^2 + t$. 

Throughout this section, we denote $x^2 + c$ by $f_c$ and $x^2 + t$ by $f_t$. Note that for $c \in \overline{\mathbb{F}}_p$, the forward orbit $\{f_c^n(0) : n = 1, 2, \dots\}$ of 0 is contained in $\mathbb{F}_p(c)$. Clearly 0 is periodic if and only if its backward orbit has points in common with its forward orbit. We thus let $f_c^{-n}(0) = \{b \in \overline{\mathbb{F}}_p : f_c^n(b) = 0\}$ and consider the sets
\[ \mathcal{I}_n = \{c \in \overline{\mathbb{F}}_p : f_c^{-n}(0) \cap \mathbb{F}_p(c) \neq \emptyset\}. \]
These sets are useful because they furnish successively better ``approximations'' of $\mathcal{H}(\overline{\mathbb{F}}_p)$, as we now show:

\begin{proposition}
For each $n \ge 1$, we have $\mathcal{I}_n \supseteq \mathcal{I}_{n+1}$. Moreover, $\mathcal{H}(\overline{\mathbb{F}}_p) = \bigcap_{n \ge 1} \mathcal{I}_n$.
\end{proposition}

\begin{proof}
Let $c \in \mathcal{I}_{n+1}$, and take $b \in \mathbb{F}_p(c)$ such that $f_c^{n+1}(b) = 0$. Then $f_c^n(f_c(b)) = 0$ and $f_c(b) \in \mathbb{F}_p(c)$, whence $c \in \mathcal{I}_n$. To show the second statement, the inclusion $\subseteq$ holds since one can follow the cycle containing 0 backwards to obtain a $n$th preimage in $\mathbb{F}_p(c)$ for all $n$. The reverse inclusion follows since $\mathbb{F}_p(c)$ is finite, so $c \in \mathcal{I}_n$ for all $n$ implies the backwards orbit of 0 under $f_c$ eventually intersects itself, implying that 0 is periodic. For details, see [14, Proposition 3.1].
\end{proof}

In light of Proposition 18, to show $\delta(\mathcal{H}(\overline{\mathbb{F}}_p)) = 0$, we need only show ${\displaystyle \lim_{n \rightarrow \infty} \delta(\mathcal{I}_n) = 0}$ (see [14, Proposition 3.2]). Now $f_c^{-n}(0) \cap \mathbb{F}_p(c) \neq \emptyset$ is equivalent to the factorization of $f_c^n(x)$ over $\mathbb{F}_p(c)$ having a linear factor. This in turn is equivalent to $f_t^n(x)$ having a linear factor modulo $(\pi_c)$, where $\pi_c \in \mathbb{F}_p[t]$ is the minimal polynomial of $c$. Since membership in $\mathcal{I}_n$ depends only on properties of $\pi_c$, it follows that $\mathcal{I}_n$ is invariant under the action of $\text{Gal}(\overline{\mathbb{F}}_p/\mathbb{F}_p)$. We thus define
\[ I_n = \{\mathfrak{p} \in \mathbb{F}_p[t] : f_t^n \text{ mod } \mathfrak{p} \text{ has at least one linear factor}\}. \]
It now follows from Lemma 14 that $\delta(\mathcal{I}_n) = \Delta(I_n),$ provided that the latter exists. 

Finally, we observe that if $G_n := G_{n,f_t}(0)$ is the $n$th quotient of the image of the arboreal representation mentioned at the beginning of this section, then by the Chebotarev density theorem,
\[ \Delta(I_n) = \frac{1}{\#G_n} \cdot \#\{g \in G_n : g \text{ fixes at least one point in } U_n\}. \]
Hence to show ${\displaystyle \lim_{n\rightarrow\infty}\Delta(I_{n})=0}$ (which by the previous paragraph implies $\delta(\mathcal{H}(\overline{\mathbb{F}}_{p}))=0)$, it is enough to establish $\mathcal{F}(G_{f_{t}}(0))=0.$ To accomplish this, we use Theorem 2 and the following theorem, whose proof follows closely that of Theorem 5:

\begin{theorem}
Consider the arboreal representation given by $K=\mathbb{F}_{p}(t)$, $V=\mathbb{A}^{1}$, $\alpha=0$, and $\phi=x^{2}+t.$ The group $H_{n}$ is maximal for all squarefree $n$.
\end{theorem}

\begin{proof}
As in the proof of Theorem 5, we must examine the critical orbit of $f_{t}$, namely the sequence $c_{n}=f_{t}^{n}(0)$, which begins $t, t^{2}+t, t^4 + 2t^3 + t^2 + t, \dots.$ We can show that $\{c_{n}\}$ is a rigid divisibility sequence (see p. 7 for definition) in exactly the same manner as the proof of Theorem 5. Also similarly to section 4, we have that $\text{Disc } f_{t}^{n}$ is a square times $c_{n}$ (see [14, Proposition 6.3]). Since the $c_{n}$ form a rigid divisibility sequence and $\text{ord}_{t}(c_{1})=1$, it follows that $\text{ord}_{t}(c_{n})=1$ for all $n$, whence $\text{Disc } f_{t}^{n}$ is not a square for all $n$.

To show that $H_{n}$ is maximal for given $n$, we must establish an analogue of Theorem 4. To do this, we take the primitive part of each $c_{n}$, in precisely the same way as one does for the cyclotomic polynomials:
\begin{equation}
\Phi_{n} := \prod_{d|n}(c_{d})^{\mu(n/d)} \in K.
\end{equation}
Again analogously to the case of cyclotomic polynomials, one can show that the $\Phi_{n}$ are in fact in $\mathcal{O}$, and are also pairwise relatively prime [14, Proposition 6.2]. Finally, one can show that $H_{n}$ is maximal if and only if $\Phi_{n}$ is not a square in $K$ [14, Theorem 6.5] (c.f. [32, Theorem, p. 16]). The final step in the proof is to calculate $\deg \Phi_{n}$, which one sees immediately from (12) is odd if $\mu(n)=1$.
\end{proof}

We have thus answered Motivating Question 4:
\[ \delta(\mathcal{H}(\overline{\mathbb{F}}_{p})) = 0 \]

More is likely true than is proven in Theorem 19:
\begin{conjecture}
Let $f_{t}=x^{2}+t \in K[x]$. Then $G_{f_{t}}(0)$ is all of $\text{Aut}(T)$.
\end{conjecture}

As evidence for this conjecture, computations have shown that $\Phi_{n}$ is not a square for $n \le 2000$, and thus $G_{n} \cong \text{Aut}(T_{n})$ for all $n \le 2000$.

\section{Analogies between arboreal representations and $p$-adic representations}

\subsection{The image of an arboreal representation}

In the previous four sections of this survey, we have been concerned with density information that can be obtained by showing that $G_{\phi}(\alpha)$ is not too small (thanks mainly to Theorem 2) or as large as possible given certain natural restrictions (section 5). In this section we take a broader view, and examine what is known and conjectured about the size of $G_{\phi}(\alpha)$ for general $V, \phi, \alpha$.

In the case of linear $l$-adic representations arising from algebraic geometry, determining the image has been both a difficult and fruitful problem. Much work has gone into showing that the image of the representation is large, specifically of finite index in the appropriate $p$-adic Lie group. One exception is the CM case, i.e. when the multiplication-by-$\ell$ map commutes with additional morphisms (see e.g. [26]). In the arboreal setting, it is natural to ask a similar question. An additional case we must exclude is when $V$ is a curve and $\phi$ is critically finite, i.e. the forward image of the branch locus $B_{\phi}$ is a finite set. When this occurs, one can show that $G_{\phi}(\alpha)$ is finitely generated as a profinite group (see e.g. [1, Theorem 1.1]). This immediately implies it cannot have finite index in $\text{Aut}(T_{\phi}(\alpha))$. However, in the absence of this phenomenon or $\phi$ commuting with some set of other morphisms, it is reasonable to expect the image of the arboreal representation to be large.

\begin{question}
Let $K$ be a global field or a function field of characteristic 0, $V$ be a curve, $\phi: V \rightarrow V$ be a finite morphism defined over $K$, and $\alpha \in V(K)$ be such that $T_{\phi}(\alpha)$ is the complete $(\deg \phi)$-ary rooted tree. Suppose $\phi$ is not critically finite and does not commute with any other morphism defined over $K^{sep}$. Must $G_{\phi}(\alpha)$ have finite index in $\text{Aut}(T_{\phi}(\alpha))$?
\end{question}

\begin{remark}
One may ask a similar question for higher-dimensional $V$, but in that case the definition of critically finite maps, and the ramification properties of extensions corresponding to them, are not well-studied.
\end{remark}

There are only a few special cases where the answer to Question 21 is known. We give a brief summary here; see [6, Section 2] and [4] for more detailed discussions. Question 21 is not even fully resolved in the case $K=\mathbb{Q}$, $V=\mathbb{A}^{1}$, $\alpha=0$, and $\phi=x^{2}+a$ for $a\in\mathbb{Z}$, which was one of the subjects of section 4. The main result is that of Stoll [32], who showed that $G_{\phi}(\alpha)$ is all of $\text{Aut}(T_{\phi}(\alpha))$ for $a>0, a\equiv1,2 \pmod{4}$ and for $a<0, a\equiv0 \pmod{4}$. We remark that the same statement does not hold for all $a\in\mathbb{Z}$ such that $x^{2}+a$ is critically infinite and has all iterates irreducible, as illustrated by $a=3$ (see remark following Theorem 4). The other principal result is that for the families $\phi=x^{2}+ax-a$ ($a \notin \{-2,2,4\}$) and $\phi=x^{2}-ax-1$ ($a \notin \{0,2\}$), $G_{\phi}(\alpha)$ has finite index in $\text{Aut}(T_{\phi}(\alpha))$ (see remark following Theorem 6 or that following [13, Theorem 1.1]). We note that in the special case $\phi=x^{2}-x+1$ Odoni showed that $G_{\phi}(\alpha)$ is all of $\text{Aut}(T_{\phi}(\alpha))$ [21].

In general, we are very far from resolving Question 21 in the case $K=\mathbb{Q}$, $V=\mathbb{A}^{1}$, $\alpha=0$, and $\phi$ is an arbitrary quadratic polynomial. It is not even known that $G_{\phi}(\alpha)$ cannot be finitely generated. We propose the following conjecture, which appears in [4]:

\begin{conjecture}[Strong Dynamical Wieferich Prime Conjecture]
Let $b\in\frac{1}{2}\mathbb{Z}$ and $\phi\in\mathbb{Z}[x]$ be separable and quadratic such that $\{\phi^{n}(b):n=1,2,...\}$ is infinite. Then for all but finitely many $n$ there exists a prime $p$ with $v_{p}(\phi^{n}(b))$ odd and $v_{p}(\phi^{m}(b))=0$ for all $m<n$.
\end{conjecture}

A few remarks are in order. By Theorem 4, Conjecture 22 implies that $G_{\phi}(\alpha)$ has finite index in $\text{Aut}(T_{\phi}(\alpha))$ for all critically infinite quadratic $\phi\in\mathbb{Z}[x]$ all of whose iterates are irreducible. As in the remark following Conjecture 8, this set of polynomials is quite a large subset of quadratic polynomials. Moreover, Conjecture 22 implies Conjecture 8. As for the name of Conjecture 22, recall that a Wieferich prime $p$ is one satisfying $2^{p-1}\equiv 1 \pmod{p^{2}}$. This condition is equivalent to the following: let $a_{n}=2^{n}-1$, and let $n_{p}$ be the smallest index such that $p \mid a_{n_{p}}$; then $p^{2} \mid a_{n_{p}}$. Currently only two Wieferich primes are known, although even the statement that their complement is infinite remains a conjecture (see e.g. [28]). A reasonable analogue of this conjecture in the dynamical setting would be that given an unbounded sequence $\{\phi^{n}(b):n=1,2,...\}$ there exist infinitely many $p$ such that $v_{p}(\phi^{n}(b))=1$ for some $n$ but $v_{p}(\phi^{m}(b))=0$ for all $m<n$. Conjecture 22 represents a significant strengthening of this, albeit with only the stipulation that $v_{p}(\phi^{n}(b))$ be odd.

The size of $G_{\phi}(\alpha)$ is somewhat better understood in the case that $K$ is a function field of characteristic 0. Odoni [20, Theorem 1] considered the situation of iterates of the generic polynomial of degree $d$, i.e. $K=k(x_{0},...,x_{d})$ ($k$ a field of characteristic 0), $V=\mathbb{A}^{1}$, $\alpha=0$ and $\phi=x_{d}z^{d}+...+x_{1}z+x_{0}$. He showed that in this case $G_{\phi}(\alpha)$ is all of $\text{Aut}(T_{\phi}(\alpha))$. It follows from Hilbert's irreducibility theorem that for all but a thin set of polynomials $\phi \in \mathbb{Z}[x]$ with all iterates irreducible, the group $G_{n,\phi}(0)$ is all of $\text{Aut}(T_{n,\phi}(0))$. However, this does not resolve Question 21 in any specific cases. 

Finally, we remark that in the case that $k=\mathbb{C}(t)$, $V=\mathbb{P}^{1}$, $\alpha=t$ and $\phi \in \C(z)$, the group $G_{\phi}(\alpha)$ is the iterated monodromy group of $\phi$ [19, 6.4.2], and thus is a self-similar subgroup of $\text{Aut}(T_{\phi}(\alpha))$ which is the closure of a group generated by a (not necessarily finite) automaton [19, 5.2]. In this case it is possible to show that if $\phi$ is a polynomial with the property that $\phi^{n}(c) \neq \phi^{m}(c)$ for all $n \neq m$ and every critical point $c$ of $\phi$, then $G_{\phi}(\alpha)$ is all of $\text{Aut}(T_{\phi}(\alpha))$, which implies the same result for $k = \mathbb{Q}(t)$.

\subsection{The trace of Frobenius and settled polynomials}
Another prominent feature of the theory of $\ell$-adic representations of the absolute Galois group $G_{K}$ of a global field $K$ is the information that can be extracted about Frobenius conjugacy classes at primes of $K$. For instance, to a cuspidal eigenform $f$ of weight $k$ and Nebentypus $\epsilon$, one may associate a representation $\rho$, an idea that goes back to Shimura [27] and others about 40 years ago. The characteristic polynomial of the image of a Frobenius element at $q$ under the $\ell$-adic representation is $x^2 - a_qx + \epsilon(q)q^{k-1}$, where $a_q$ is the $q$th coefficient of $f$. Taking the product of the reciprocals of these characteristic polynomials for varying $q$ with $x=q^{-s}$ and a modified version for primes $q$ dividing the level of $f$ produces the $L$-series $L(s,\rho)$, which turns out to be independent of both $\rho$ and $\ell$ and so is denoted $L(s,f)$. 

We thus wish to associate a conjugacy invariant to the image of $\textrm{Frob}_q$ in the arboreal setting, analogously to the trace $a_q$ in the linear setting. We restrict ourselves to the case $K=\mathbb{Q}$, $V=\mathbb{A}^{1}$, $\alpha=0$, and $\phi$ a quadratic polynomial, where already little is known. A primary source of information about $\textrm{Frob}_q$ is that the cycle decomposition of its action on the roots of $\phi^n$ is given (except for finitely many $q$) by the degrees of the irreducible factors of $\phi^n \pmod{q}$. We thus introduce the notion of settledness, which figures in our proposed analogue of the trace of Frobenius:

\begin{definition}
Let $q$ be a prime. Given a quadratic polynomial $f \in \mathbb{F}_q[x]$, a polynomial $h \in \mathbb{F}_q[x]$ is called $f$-stable if for every $n \ge 0$, $h \circ f^n$ is irreducible. For a given $n$ let $g_1, ..., g_r$ denote the $f$-stable factors of $f^n$ and $s_n$ the sum of their degrees. The polynomial $f \in \mathbb{F}_q[x]$ is called \textrm{settled} if the limit of $s_n/2^n$ as $n \rightarrow \infty$ is 1.
\end{definition}

We study settledness (over any finite field) in [5], where it is shown that if $h(a)$ is not a square for every $a$ in the critical orbit of $f$, then $h$ is $f$-stable. We also conjecture there that every irreducible quadratic $f$ with coefficients in a finite field is settled and give computational evidence in support. We now define settled elements of $\text{Aut}(T_{\phi}(\alpha))$. Denote $T_{\phi}(\alpha)$ by $T$ and denote by $T_n$ the truncation of $T_{\phi}(\alpha)$ to the first $n$ levels.

\begin{definition}
Let $T$ be the infinite complete binary rooted tree. Suppose an element $\sigma \in \Aut(T)$ has image $\sigma_n \in \Aut(T_n)$. A cycle of $\sigma_n$ of length $2^k$ is called stable if it is mapped to by a cycle of $\sigma_r$ of length $2^{k+r-n}$ for all $r > n$. Let the sum of the lengths of the stable cycles of $\sigma_n$ be $s_n$. Then $\sigma$ is called settled if the limit of $s_n/2^n$ as $n \rightarrow \infty$ is 1.
\end{definition}

For example, an element of $\text{Aut}(T)$ that acts as a single cycle on every level is settled. Such an element is called an adding machine (see [19, Chapter 1]). Settled elements consist of a proliferation of adding machines on subtrees of $T$. 

The main conjecture of [5] mentioned above implies that if $f$ is an irreducible quadratic polynomial in $\mathbb{Z}[x]$, then the Frobenius conjugacy classes in $G_f := G_f(0)$ consist of settled elements. We can thus associate to a prime $q$ a certain (possibly infinite) partition of 1. Namely, if $\sigma \in \text{Aut}(T)$ is in the class of $\textrm{Frob}_q$, and the stable parts of $\sigma_n$ have degrees $d_1, ..., d_r$ then $d_1/2^n + d_2/2^n + ... + d_r/2^n$ is the initial segment of what as $n \rightarrow \infty$ becomes a partition of 1. This might be thought of as analogous to a local zeta function. The question arises as to whether this partition has a finitely expressible generating function. In [5], it is conjectured that a Markov process approximates the factorization of the iterates $f^n \pmod{q}$ and the rate at which this process converges might alternatively be associated to $q$. One is then led to ask how to convert these numbers into useful analogues of $L$-series, although at the present there appears to be no clear answer.

Finally, since by Chebotarev's density theorem the Frobenius elements are dense, it follows that the settled elements in $G(f)$ are dense in $G(f)$. We call such a subgroup of $\text{Aut}(T)$ \textit{densely settled}. It is easy to see that settled elements are rare (of density zero) in $\text{Aut}(T)$, in analogy to the characteristic polynomial of Frobenius having algebraic coefficients in the $\ell$-adic case. We thus propose a study of densely settled subgroups of $\text{Aut}(T)$. There certainly exist subgroups of $\text{Aut}(T)$ that fail to be densely settled, for example because they have too much torsion (torsion elements are never settled). A group is densely settled if and only if it has a densely settled subgroup of finite index, and so we need consider densely settled groups only up to commensurability. In general, we ask for a classification of them up to commensurability (in analogy to how $p$-adic Lie groups are classified up to commensurability by their Lie algebra). For example, which groups defined by automata are densely settled?

\end{document}